\documentclass[lm,a5paper,twoside,fleqn,leqno,gray]{pre} 

\usepackage[applemac]{inputenc}
\usepackage{math}

\theoremstyle{plain}
\newtheorem*{main}{Differentiable \kam Theorem}
\newtheorem*{step}{Step Lemma}
\newtheorem*{iter}{Iterative Lemma}

\let\Ps  \Psi
\let\Th  \Theta
\def\Phh{\hat\Phi}
\def\Psh{\hat\Psi}
\def\Psp{\Psi_\pl}
\def\Yp {Y_\pl}
\def\Pp {P_\pl}
\def\Qn {Q^\pl}
\def\Qp {Q_\pl}
\def\DP {\Dl P}
\def\Pr {\Pi}
\def\Pro{\Pr_0}
\def\Prs{\Pr_1}

\def\no#1{\abs*{#1}_0}
\def\ns#1{\abs*{#1}_s}
\def\nns#1{\nn*{#1}_s}
\def\na#1{\n*{#1}_{\al s}}
\def\nna#1{\nn*{#1}_{\al s}}
\def\nnaa#1{\nn*{#1}_{\al^2s}}
\def\sn {{s_\nu}}
\def\nup{{\nu+1}}

\def\intm{\mkern2mu\mathop\nshortmid\mathchoice{\mkern-14.5mu}{\mkern-12.6mu}{\bullet}{\bullet}\int}
\def\sumk{\sum_{k\in\Zn}}
\def\suml{\sum_{l\in\Zn}}
\def\veee{\,\vee\,}
\def\r   {r}
\def\1{_1}

\let\Newpage\relax

\title   {KAM below $\mathbf C^n$} 
\author  {Jürgen Pöschel}
\date    {Version 1.1, November 2020}
\runtitle{KAM below $C^n$} 

\begin{document}

\maketitle

\begin{abstract}
We consider the \kam theory for rotational flows on an \m{n}-dimen\-sional torus. We show that if its frequencies  are diophantine of type $n-1$, then Moser's \kam theory with parameters applies to small perturbations of weaker regularity than~$C^n$. Derivatives of order $n$ need not be continuous, but rather \m{L^2} in a certain strong sense. This disproves the long standing conjecture that $C^n$ is the minimal regularity assumption for \kam to apply in this setting while still allowing for Herman's $C^{n-\ep}\!$-counterexamples.
\end{abstract}

We consider one of the model problems of \kam theory, namely the perturbation of a constant vectorfield 
\[
  N = \sum_{i=1}^n \om_i \del_i
\]
on an \m{n}-dimensional torus given by $n$ frequencies
\[
  \om = (\om_1,\dots,\om_n).
\]
The flow of this vector field is very simple, being the superposition of uniform rotations on each circle of the torus.

Upon perturbation this simple flow is usually destroyed, and chaotic behaviour may set in. However, it is one of the fundamental results of \kam theory that if the frequencies are strongly nonresonant, and if the perturbation $P$ is sufficiently smooth and sufficiently small, then there exists a modifying constant vector field $Y$ so that $N+P-Y$ is indeed conjugate to~$N$. Or to put it succinctly, \lq\kam applies\rq.

We address the question of how smooth that perturbation has to be.
As is well known the answer depends on the diophantine properties of the frequencies~$\om$.
So we assume that $\om$ is diophantine with exponent $\ta$, that is,
\[
  \n{\ipr{k,\om}} \ge \al\n{k}^{-\ta}, \qq  0\ne k\in\Zn,
\]
with some $\al>0$. 
This holds for almost all frequencies when $\ta>n-1$, and on a non-empty zero set when $\tau=n-1$. 

Now consider a perturbation~$P$ of~$N$. Writing
\[
  P = \sum_{0\ne k\in\Zn} p_ke_k,
  \qq
  e_k=\e^{\i\ipr{k,x}},
\]
we measure the size of $P$ with respect to the norms
\[
  \nn{P}_{r,b}
  \defeq \sum_{\nu\ge0} \pas3{ \sum_{b^{\nu-1}< \n{k}\le b^\nu} \n{p_k}^2 \n{k}^{2r} } ^{1/2}.
\]
where $b\ge2$ is an integer.
These norms interpolate between the standard weighted \m{\ell^1}- and \m{\ell^2}-norms
\[
  \nn{P}_{r,1}
  \defeq \sum_{k\ne0} \n{p_k} \n{k}^{r},
  \qq
  \nn{P}_{r,\iny}
  \defeq \pas3{ \sum_{k\ne0} \n{p_k}^2 \n{k}^{2r} }^{1/2}. 
\] 
Indeed, as one easily verifies,
\mmx
$\nn{P}_{r,1} \ge \nn{P}_{r,b} \searrow \nn{P}_{r,\iny}$ as $b\to\iny$.
\[
  \nn{P}_{r,1} \ge \nn{P}_{r,b} \searrow \nn{P}_{r,\iny},
  \qq
  b\to\iny.
\]

\begin{main}
Suppose the frequencies $\om$ of~$N$ are diophantine with exponent~$\ta$. Then KAM applies to perturbations $P$ of\/~$N$, whenever $\nn{P}_{\tau+1,b}$ is sufficiently small for some integer $b\ge2$.
\end{main}

This result is new and, it seems, unexpected. 
Consider the case of $n$ degrees of freedom and $\ta=n-1$. 
It was conjectured for a long time, and sometimes even stated as fact, that $C^n$ is the minimal regularity requirement for \kam to apply in this case \cite{Alb,Che,Her}. 
This, however, is not true.
For $\nn{P}_{n,b}$ to be small for any $b\ge2$, the \m{n-1}-derivative of $P$ need not be continuous, but only of strong \m{L^2}-type in the sense of this norm.

We point out that this result does not contradict the counterexamples of class $C^{n-\ep}$ given by Herman~\cite{Her}. 
It rather fits precisely into the gap between the class of $C^n$-perturbations, for which \kam was known to apply, and the class of $C^{n-\ep}$-counterexamples of Herman.

The rest of this paper is devoted to the proof of this theorem.
We combine a succession of coordinate transformations familiar in \kam theory with the analytic smoothing approach of Moser~\cite{Mo67}. This way no analytic perturbation theorem as an intermediate step is required.
The convergence speed of this procedure is not of Newton type, but rather slow in the spirit of Rüssmann~\cite{R-09,P-11}. Indeed, at each step of the iteration, we do not linearize the conjugacy equation under consideration, but rather solve a nonlinear equation with the help of the Brouwer fixed point theorem. 


\Newpage

\section{Norms}

The proof of the main theorem makes use of the interplay of two different norms for analytic functions and maps. For a function or map $f$ on the torus with Fourier series expansion $f=\sumk f_ke_k$ and $s\ge0$ we set
\[
  \ns{f} \defeq \sumk \n{f_k} \e^{s\n{k}},
\]
and, with $U_s\colon \n{\Im z}\le s$,
\begin{align*}
  \nns{f}^2
  \defeq \intm_{U_s} \n{f}^2
  &\defeq \frac{1}{\n{U_s}} \int_{U_s} \n{f}^2 \\
  &= \frac{1}{(4\pi s)^n} \int_{[-s,s]^n}\int_{[-\pi,\pi]^n}  \n{f(x+\i y)}^2 \dx\dy.
\end{align*}
This norm also has a representation with particular weights.

\begin{lem} \label{equi}
For $f = \sum_k f_ke_k$,
\[
  \nns{f}^2
  = \sumk \smash[b]{\n{f_k}^2} w_k(2s)
\]
with
\[
  w_k(t) = \smash[t]{\prod_{i=1}^n \frac{\sinh(tk_i)}{tk_i}}
\]
and the convention that $t\inv\sinh t = 1$ for $t=0$.
\end{lem}

\begin{proof}
The general term of $\n{f}^2$ at $z=x+\i y$ is
\[
  f_k\e^{\i\ipr{k,z}} \fb_{\!l}\e^{-\i\ipr{l,\zb}}
  = f_k\fb_{\!l}\, \e^{\i\ipr{k-l,x}} \e^{-\ipr{k+l,y}}.
\]
Its mean integral over $U_s$ vanishes whenever $k\ne l$. For $k=l$ we get
\[
  \frac{1}{(2s)^n} \int_{[-s,s]^n} \n{f_k}^2 \e^{-\ipr{2k,y}}\dy
  = \n{f_k}^2 \prod_{i\colon k_i\ne0} \frac{\sinh(2sk_i)}{2sk_i},
\]
which gives the claimed identity. 
\end{proof}

We note that for constant functions $Z$,
\mmx
$\n{Z} = \no{Z} = \ns{Z} = \nns{Z}$ for all $s\ge0$.
\[
  \n{Z} = \no{Z} = \ns{Z} = \nns{Z}, \qq s\ge0.
\]
Moreover,
\[
  \n{f}_{U_s} \defeq \sup_{U_s} \n{f} \le \ns{f},
\]
so we have
\$
  \nns{fg} \le \ns{f}\nns{g}
$
and in particular
$
  \nns{f}\le\ns{f}
$.
Finally, the exponential weights are submultiplicative, so we have the Banach algebra property
\$
  \ns{fg}\le\ns{f}\ns{g}.
$ 
The usual Neumann series argument then gives us

\begin{lem} \label{inv}
Let $\ph=I+\phh$ be a self map the \m{n}-torus. 
If\/ $\mu = \ns{D\phh}<1$, then $\ph$ is a diffeomorphism, and for its inverse $\ps=I+\psh$ one has
\[
  \ns{D\psh} \le \frac{\mu}{1-\mu},
  \qq
  \ns{D\ps} \le \frac{1}{1-\mu}.
\]
\end{lem}

A central role is played by the behaviour of these norms under coordinate transformations.

\begin{lem} \label{comp}
For a diffeomorphismus $\ph$ of the \m{n}-torus 
extending to~$U_s$, 
\[
  \nns{f\cmp\ph} \le \lm \nn{f}_{\lm s},
  \qq
  \lm \defeq \n*{D\ph}_{U_s} \veee \n*{D\ph\inv}_{\ph(U_s)} \ge 1,
\]
where $\vee$ denotes the maximum operator.
Moreover, if If\/ $\ns{\ph-I}\le\sg$, then
\[
  \ns{f\cmp\ph} \le \n{f}_{s+\sg}.
\]
\end{lem}

\begin{proof}
We have $\ph(U_s) \subset U_{\lm s}$ by the reality of $\ph$ and the bound on $D\ph$. Therefore,
\begin{align*}
  \intm_{U_s} \n{f\cmp\ph}^2
  &\le \intm_{\ph(U_s)} \n{f}^2\n*{D\ph\inv}^2 \\
  &\le \n*{D\ph\inv}_{\ph(U_s)}^2 \intm_{\ph(U_s)} \n{f}^2 
  \le \lm^2 \intm_{U_{\lm s}} \n{f}^2,
\end{align*}
proving the first claim. Writing $\ph=I+\phh$ we have
$e_k\cmp\ph = e_k\cd e_k\cmp\phh$
and thus
\mmxx
\begin{align*}
  \ns{e_k\cmp\ph}
  &\le \ns{e_k}\ns{e_k\cmp\phh} \\
  &\le \ns{e_k}\exp(\ns{\ipr{k,\phh}})
  \le \e^{s\n{k}}\e^{\sg\n{k}}
  = \e^{(s+\sg)\n{k}}.
\end{align*}
\[
  \ns{e_k\cmp\ph}
  \le \ns{e_k}\ns{e_k\cmp\phh}
  \le \ns{e_k}\exp(\ns{\ipr{k,\phh}})
  = \e^{(s+\sg)\n{k}}.
\]
For general $f$ we thus obtain
\[
  \ns{f\cmp\ph}
  \le \sumk \n{f_k}\ns{e_k\cmp\ph}
  \le \sumk \n{f_k} \e^{(s+\sg)\n{k}}
  = \n{f}_{s+\sg}.
  \qed
\]
\end{proof}

We also have a Cauchy inequality connecting the two norms $\ns\cd$ and $\nns\cd$.

\begin{lem} \label{cauchy}
For compatible maps $f$ and $\ph$ on the \m{n}-torus,
\[
  \nna{\Df\cd\ph} \le \frac{1}{\e\al^{n/2}} \frac{1}{(1-\al)s} \nns{f}\ns{\ph}, \qq 0<\al<1.
\]
\end{lem}

\begin{proof}
It is easy to verify that
$\dfrac{\sinh t}{t}\e^{-t}$ is decreasing for all~$t$.
Hence,
\[
  \frac{\sinh s}{s} \le \frac{\sinh(s-t)}{s-t} \e^t, \qq 0\le t\le s.
\]
It follows that for all $k$ and~$l$ we have the submultiplicativity property
\[
  •ewkl
  w_k(s) \le w_{k-l}(s)\e^{s\n{l}}, \qq s\ge0,
\]
connecting the weights of $\nns\cd$ and $\ns\cd$.

Now consider the Fourier series expansion
\[
  \frac1\i \Df\cd\ph
  = \sumk \ipr{k,\ph}f_ke_k 
  = \sum_{k,l\in\Zn} \ipr{k-l,\ph_l}f_{k-l}e_l.
\]
By the triangle inequality,
\begin{align*}
  \nna{\Df\cd\ph}
  &\le \pas3{ \sumk \n2{\suml \ipr{k-l,\ph_l}f_{k-l}}^2 w_k(2\al s) }^{1/2} \\
  &\le \suml \pas3{ \sumk \n{k-l}^2\n{\ph_l}^2\n{f_{k-l}}^2 w_k(2\al s) }^{1/2}.
\end{align*}
With inequality~\eqref{ewkl} and a subsequent re-indexing it follows that
\begin{align*}
  \nna{\Df\cd\ph}
  &\le \suml \n{\ph_l}\e^{\al s\n{l}} \pas3{ \sumk \n{k}^2\n{f_k}^2w_k(2\al s) }^{1/2} \\
  &\le \sup_{k\in\Zn} \n{k} \pas{\frac{w_k(2\al s)}{w_k(2s)}}^{1/2} \nns{f}\ns{\ph}. 
\end{align*}
Now, with $\dfrac{\sinh s}{\sinh t} \le \e^{s-t}$ for $s\le t$ we get
\[
  •e-ww
  \frac{w_k(s)}{w_k(t)}
  = \pas{\frac st}^n \prod_{1\le i\le n} \frac{\sinh(sk_i)}{\sinh(tk_i)}
  \le \pas{\frac st}^n \e^{(s-t)\n{k}},
\]
so the supremum is bounded by
\[
  \sup_{t\ge0} \frac{t}{\al^{n/2}}\e^{(\al-1)st}
  = \frac1{s\al^{n/2}} \sup_{t\ge0} r\e^{(\al-1)r}
  = \frac{1}{\e\al^{n/2}} \frac{1}{(1-\al)s}.
\]
This proves the claim.
\end{proof}

\section{Small Divisors and Cut Offs}  \label{s-sd}

Next we consider the solution of the typical small divisor equation
\mmx
$\del_\om\phi=f$
\[
  \del_\om\phi
  = \om_1\del_1\phi+\ldots+\om_n\del_n\phi = f
\]
with a nonresonant frequency vector~$\om$. Formally, for $f=\sum_{k\ne0} f_ke_k$, the unique solution with mean value zero is
\[
  \phi = Lf \defeq \sum_{k\ne0} \frac{f_k}{\i{\ipr{k,\om}}} e_k.
\]

\begin{lem} \label{sd}
Suppose that
\$
  \Om \defeq \max_{0<\n{k}\le K} \n{\ipr{k,\om}}\inv < \iny
$.
Then, for a tri\-go\-no\-metric polynomial $f$ of order $K$ without constant term, 
\[
  \ns{Lf}  \le C\Om \nns{f},
  \qq
  C = 2^n\e^{sK}.
\]
\end{lem}

\begin{proof}
With Cauchy-Schwarz,
\[
  \n{Lf}_0^2
  \le \sum_{0<\n{k}\le K} \frac{1}{\ipr{k,\om}^2} \, \sum_{0<\n{k}\le K} \n{f_k}^2.
\]
By an elegant estimate due to Rüssmann~\cite{R-76},
\[
  \sum_{0<\n{k}\le K} \frac{1}{\ipr{k,\om}^2}
  \le 2^{n+2}\Om^2,
\]
while the second sum is just $\nn{f}_0^2$.
Since the exponential weights in $\ns{Lf}$ are all bounded by $\e^{sK}$ and the weights in $\nns{f}$ are all $\ge1$ we obtain
\[
  \ns{Lf}
  \le \e^{sK} \n{Lf}_0 
  \le C\Om \nn{f}_0
  \le C \Om \nns{f}.
  \qed
\]
\end{proof}

\goodbreak

We also need an estimate for ultra-violent cut offs.

\begin{lem} \label{cut}
Suppose $f$ contains no Fourier coefficients up to order~$K$. Then
\[
  \nna{f} \le \al^{-n/2} \e^{(\al-1)sK} \nns{f}, \qq 0<\al\le1.
\]
\end{lem}

\begin{proof}
For
$
  f = \sum_{\n{k}\ge K} f_ke_k
$ 
we get
\[
  \nna{f}^2
  = \sum_{\n{k}\ge K} \n{f_k}^2 w_k(2\al s)
  \le \sup_{\n{k}\ge K} \frac{w_k(2\al s)}{w_k(2s)} \nns{f}^2.
\]
By the same argument as in~\eqref{e-ww}
\[
  \sup_{\n{k}\ge K} \frac{w_k(2\al s)}{w_k(2s)} \le \al^{-{n/2}} \e^{(\al-1)sK}.
  \qed
\]
\end{proof}


\Newpage

\section{Outline of Proof and Step Lemma}  

Suppose we already found a modifying term $Y$ and a coordinate transformation $\Ps$ so that
\[
  \Ps^*(N+P-Y) = N+Q,
\]
where $P$ is an analytic approximation – indeed a trigonometric polynomial – to the original smooth perturbation we are aiming at, and $Q$ is smaller than~$P$.
We then construct another modifying term $Z$ and a transformation $\Ph$ so that
\[
  •eQZ
  \Ph^*(N+Q-\Ps^*Z) = N+\Qn
\]
improves on~$Q$. Setting $\Psp = \Ps\cmp\Ph$ and $\Yp=Y+Z$ we obtain
\begin{align*}
  \Psp^*(N+P-\Yp)
  &= \Ph^*(\Ps^*(N+P-Y) - \Ps^*Z) \\
  &= \Ph^*(N+Q-\Ps^*Z) \\
  &\eqdef N + \Qn.
\end{align*}
Passing from $P$ to the next approximation $\Pp$ we arrive at
\mmxx
\begin{align*}
  \Psp^*(N+\Pp-\Yp)
  &= N + \Qn + \Psp^*(\Pp-P) \\
  &= N + \Qp,
\end{align*}
\[
  \Psp^*(N+\Pp-\Yp)
  = N + \Qn + \Psp^*(\Pp-P)
  \eqdef N + \Qp,
\]
which completes one cycle of the iterative procedure.
The following lemma describes the quantitative details of this construction.

\goodbreak

\begin{step}
Consider $\Ps^*(N+P)=N+Q$. Assume that
\[
  4\Dl\nns{Q} \le \frac14, \qq 
  \ns{D\Ps-I} \le \frac17,
\]
where $\Dl = CK\Om$ with $C = 2^n\e^{sK}$ and
\$
  \Om = \max_{0<\n{k}\le K} \n{\ipr{k,\om}}\inv
$.
Also assume that $sK\ge (4/3)^{n/2}$.
Then there exists a unique modifying term $Z$ and a unique coordinate transformation $\Ph=I+\Phh$ with
\[
  \Dl\no{Z} \veee K\nn*{\Phh}_{s/2} \le 4\Dl\nns{Q},
\]
so that
$
  \Ph^*\Ps^*(N+P-Z) = N+\Qn
$
with
\[
  \nn*{\Qn}_{s/4} \le 4\nn*{(I-\Pr)Q}_{s/2},
\]
where $\Pi$ denotes truncation of Fourier series at order~$K$.
\end{step}


\begin{proof}
Consider equation \eqref{eQZ} which with $\Ph=I+\Phh$ is equivalent to
\[
  •e1
  D\Phh\cd N + D\Ph\cd \Qn = (Q-\Ps^*Z)\cmp\Ph.
\]
Instead of this functional equation we solve the finite dimensional equation
\[
  •e2
  D\Phh\cd N = \Pr((Q-\Ps^*Z)\cmp\Ph),
\]
where $\Pr$ denotes truncation of a Fourier series at order~$K$. Writing
\[
  \Ps^*Z = Z-\Th Z, \qq \Th = D\Ps\inv(D\Ps-I),
\]
its right hand side becomes
$\Pr((Q-\Ps^*Z)\cmp\Ph) = \Pr((Q+\Th Z)\cmp\Ph) - Z$
so that this equation amounts to
\[
  •e3
  D\Phh\cd N + Z
  = \Pr\, T(Z,\Phh)
\]
with the nonlinear operator
\[
  T(Z,\Phh) \defeq (Q+\Th Z)\cmp(I+\Phh).
\]

Here, $D\Phh\cd N$ 
is the familiar linear differential operator $\del_\om\Phh$ giving rise to small divisors. 
Its inverse on the space of trigonometric polynomials of order~$K$ with vanishing mean value is the operator $L$ considered in Lemma~\ref{sd}.
Thus, a solution of equation~\eqref{e3} is a fixed point of the mapping
\[
  •e4
  Z_1 = \Pro T(Z,\Phh), 
  \qq
  \Phh\1 = L\,\Prs T(Z,\Phh),
\]
where $\Pro$ denotes the operator of taking the mean value over the \m{n}-torus and $\Prs=\Pr-\Pro$.
This we solve with the Banach contraction principle.

To this end let $\al=1/2$ for brevity and consider the ball
\[
  \Bs:\q\Dl\no{Z} \veee K\na{\Phh} \le 4\Dl\nns{Q} \le \frac14
\]
within the space of constant terms times trigonometric polynomials of order~$K$ without constant terms.
By assumption on $\Ps$ and Lemma~\ref{inv}, 
\[
  \ns{\Th}
  = \ns{D\Ps\inv}\ns{D\Ps-I}
  \le \frac{1/7}{1-1/7}
  = \frac16.
\]
Hence, $\nns{\Th Z} \le \ns{\Th Z} \le \ns{\Th}\no{Z} \le \nns{Q}$ and thus
\[
  \nns{Q+\Th Z} \le 2\nns{Q}.
\]
Similarly, $\na{D\Phh} \le K\na{\Phh} \le 1/4$ and therefore
\[
  •edph
  \na{D\Ph} \veee \na{D\Ph\inv} \le \frac32.
\]
So with Lemma~\ref{comp}
\mmxx
\begin{align*}
  \nna{T(Z,\Phh)}
  &= \nna{(Q+\Th Z)\cmp(I+\Phh)} \\
  &\le 2\nns{Q+\Th Z} 
  \le 4\nns{Q}.
\end{align*}
\[
  \nna{T(Z,\Phh)}
  = \nna{(Q+\Th Z)\cmp(I+\Phh)}
  \le 2\nns{Q+\Th Z} 
  \le 4\nns{Q}.
\]
It follows that 
\begin{align*}
  \Dl\no{Z_1} &= \Dl\no{\Pro T(Z,\Phh)} \le 4\Dl\nns{Q}, \\
  K\na{\Phh_1} &= K\na{L\,\Prs T(Z,\Phh)} \le CK\Om \nna{T(Z,\Phh)} \le 4\Dl \nns{Q}.
\end{align*}
So our ball $\Bs$ is mapped into itself, and continuously so. 
So the Brouwer fixed point theorem applies.

Indeed, this map is a contraction. Consider
\begin{align*}
  T(Z,\Phh)-T(Z',\Phh')
  &= (Q+\Th Z)\cmp\Ph-(Q+\Th Z)\cmp\Ph' \\
  &\qq + (\Th(Z-Z'))\cmp\Ph' \\
  &\eqdef A+B.
\end{align*}
In view of ~\eqref{edph} any map in $\Bs$ maps $U_{\al s}$ into $U_{s-\sg}$ with $\sg=s/4$. So we can apply the Cauchy estimate of Lemma~\ref{cauchy} -- here with $\al=3/4$ -- to $A$ to obtain
\begin{align*}
  \nna{A} 
  &\le \pas{\frac43}^{n/2} \frac 2s \nns{Q+\Th Z} \na{\Phh-\Phh'} \\
  &\le \pas{\frac43}^{n/2} \frac4{sK} \nns{Q}\,K\!\na{\Phh-\Phh'}
  \le 4\nns{Q}\,K\!\na{\Phh-\Phh'}
\end{align*}
by assuming that $sK\ge (4/3)^n$. Similarly, by Lemma~\ref{comp}
\[
  \nna{B}
  = \nna{(\Th(Z-Z'))\cmp\Ph'}
  \le \frac32\nns{\Th(Z-Z')}
  \le \frac14 \no{Z-Z'}.
\]
Since we assume that $4\Dl\nns{Q}\le1/4$ it follows that
\mmxx
\begin{align*}
  K\na{\Phh\1-\Phh\1'} 
  &\le \Dl\nna{A}+\Dl\nna{B} \\
  &\le \frac14\,K\!\na{\Phh-\Phh'} + \frac14\,\Dl\no{Z-Z'}.
\end{align*}  
\[
  K\na{\Phh\1-\Phh\1'} 
  \le \Dl\nna{A}+\Dl\nna{B}
  \le \frac14\,K\!\na{\Phh-\Phh'} + \frac14\,\Dl\no{Z-Z'}.
\]
Exactly the same estimate holds for $\Dl\no{Z\1-Z\1'}$. But this means that with respect to the norm
\[
  \Dl\no{Z-Z'} \veee K\na{\Phh-\Phh'}
\]
we obtain a contraction by the factor $1/2$.

The estimates for the unique fixed point $Z,\Phh$ are the same as for the ball~$\Bs$.
So it remains to bound $\Qn$. From \eqref{e1} and \eqref{e2} we deduce that
\[
  D\Ph\cd\Qn  = (I-\Pr)((Q-\Ps^*Z)\cmp\Ph).
\]
With 
$
  \nns{\Psi^*Z}
  \le \ns{D\Psi\inv}\no{Z}
  \le \nns{Q}
$,
estimate \eqref{edph} and Lemma~\ref{cut} we arrive at
\begin{align*}
  \nnaa{\Qn}
  &\le \na{D\Ph\inv\cmp\Ph} \nnaa{(I-\Pr)((Q-\Ps^*Z)\cmp\Ph)} \\
  &\le 2 \nna{(I-\Pr)(Q-\Ps^*Z)} \\
  &\le 4 \nna{(I-\Pr)Q}.
\end{align*}  
This finishes the proof of the Step Lemma.
\end{proof}

\Newpage

\section{Iteration}

We now assume the frequency vector $\om$ of the vector field $N$ to be diophantine with exponent~$\ta$. 
Scaling time, we may even assume that
\[
  \n{\ipr{k,\om}} \ge \n{k}^{-\ta}, \qq k\ne0.
\]
We recall that $\ta\ge n-1$, since otherwise no such frequencies exist.

To simplify the exposition we now assume with loss of generality that $\nn{P}_{r,b}$ is small with $b\ge4$.
For $\nu\ge0$ we set
\[
  K_\nu = b^\nu, \qq 
  s_\nu = \frac{\r}{b^\nu}, \qq
  \Dl_\nu = K_\nu^{\ta+1} = b^{\nu(\ta+1)}.
\]
Thus, 
\mmx
$s_\nu K_\nu=r$ for all $\nu\ge0$,
\[
  s_\nu K_\nu=r, \qq \nu\ge0,
\]
and we may simplify the estimates of section~\ref{s-sd}
\begin{align*}
  K_\nu\n{Lf}_\sn \veee \n{DLf}_\sn &\le 2^{n}\e^r \, \Dl_\nu \nn{f}_\sn,
  \\
  \nn{F}_{s_{\nu/2}} &\le 2^{n/2}\e^{-r/2} \, \nn{F}_\sn,
\end{align*}
for $F$ containing no Fourier coefficients up to order~$K_\nu$.
We then fix $r$ so that 
\[
  \th \defeq 4^n\e^{-r/2} \le \frac{1}{2\cd b^{\ta+1}}.
\]

We approximate the given perturbation 
$
  P = \sum_{k\ne0} p_ke_k
$ 
by the sequence of trigonometric polynomials 
\[
  P_\nu = \sum_{\n{k}\le K_\nu} p_ke_k, \qq \nu\ge0.
\]
To make the exposition more transparent with respect to the norm of $P$ we introduce the weighted \m{L^2}-norm
\[
  \nn{P}_m^2 = \sum_{k\ne0} \n{p_k}^2m_k^2.
\]

\begin{lem}
For $\DP_0=P_0$ and $\DP_v = P_\nu-P_{\nu-1}$ for $\nu\ge1$ we have
\[
  \nn{\DP_\nu}_\sn \le \frac{\e^r}{m_\nu} \nn{\DP_\nu}_m, \qq \nu\ge0,
\]
with $m_0=1$ and
\mmx
$m_\nu = \min_{\n{k}>K_{\nu-1}} m_k$ for $\nu\ge1$.
\[
  m_\nu = \min_{\n{k}>K_{\nu-1}} m_k, \q \nu\ge1.
\]
\end{lem}

\begin{proof} \let\Lm\Ks
With $w_k(s) \le \e^{s\n{k}}$ and
$\Lm_\nu \colon K_{\nu-1}<\n{k}\le K_\nu$ we get
\begin{align*}
  \nn{\DP_\nu}_\sn^2
  &= \sum_{k\in\Lm_\nu} \n{p_k}^2 w_k(2\sn) \\[-2pt]
  &\le \max_{k\in\Lm_\nu} \frac{\e^{2\sn\n{k}}}{m_k^2} \sum_{k\in\Lm_\nu} \n{p_k}^2 m_k^2 
  \le \frac{\e^{2r}}{m_\nu^2} \nn{\DP_\nu}_m^2.
  \qed
\end{align*}
\end{proof}

\goodbreak[6]

\begin{iter}
Suppose that
\mmx
$\displaystyle\en \smash{\sup_{\nu\ge0} \frac{\Dl_\nu}{m_\nu}} \le A < \iny \en$ 
\[
  \sup_{\nu\ge0} \frac{\Dl_\nu}{m_\nu} \le A < \iny,
\]
and that
\[
  \ep = \sum_{\nu\ge0} \nn{\Dl P_\nu}_{m}
\]
is sufficiently small.
Then for each $P_\nu$ there exists a modifying term $Y_\nu$ and a transformation $\Ps_\nu$ such that
\[
  \Ps_\nu^*(N+P_\nu-Y_\nu) = N + Q_\nu
\]
with
\begin{align*}
  \nn{Q_\nu}_\sn 
  &\le \ep_\nu \defeq B \sum_{\mu\le \nu} \frac{\th^{\nu-\mu}}{m_\mu} \nn*{\DP_\mu}_m,
  \\
  \n*{D\Psh_\nu}_\sn 
  &\le \dl_\nu \defeq \prod_{0\le \mu<\nu} (1+4\Dl_\nu \ep_\nu) - 1,
\end{align*}
where $B=2\e^r$.
Moreover, $\en\n{Y_{\nu+1}-Y_\nu} \le  4\ep_\nu$ and
$
  \n{D\Ps_\nup-D\Ps_\nu}_0 \le 8\Dl_\nu\ep_\nu
$.
\end{iter}

\begin{proof} 
For $\nu=0$ we can choose $Y_0=0$ and $\Ps_0=I$. Then $Q_0=P_0$, and the estimate for $Q_0$ is satisfied by the preceding lemma by choice of~$B$.
So we may proceed by induction.

By our choice of~$\th$, the sequence $\th^\nu\!\Dl_\nu$ decays geometrically so that
\[
  \sum_{\nu\ge\mu} \th^\nu\! \Dl_\nu
  \le 2\cd\th^\mu \Dl_\mu.
\]
Therefore, with the abbreviation $\rh_\mu \defeq \nn*{\DP_\mu}_m$ and $\Dl_\nu/m_\nu\le A$,
\begin{align*}
  \sum_{\nu\ge0} \Dl_\nu\ep_\nu
  = B \sum_{\mu\ge0} \frac{\th^{-\mu}}{m_\mu}\rh_\mu  \sum_{\nu\ge\mu} \th^\nu\!\Dl_\nu 
  &\le 2B \sum_{\mu\ge0} \frac{\Dl_\mu}{m_\mu} \rh_\mu \\
  &\le 2AB \sum_{\mu\ge0} \rh_\mu
   = 2AB\,\ep.
\end{align*}
It follows that for $\ep$ small enough, the smallness conditions of the Step Lemma are satisfied by $\Ps_\nu$ and $Q_\nu$ for all $\nu\ge0$.

We obtain a modifying term $Z_\nu$ and a transformation $\Ph_\nu$ with
\[
  •edph
  \no{Z_\nu} \le 4\ep_\nu,
  \qq
  K_\nu \nn*{\Phh_\nu}_{s_\nu/2} \veee \nn*{D\Phh_\nu}_{s_\nu/2} \le 4\Dl_\nu\ep_\nu,
\]
With $Y_\nup = Y_\nu+Z_\nu$ and $\Ps_\nup = \Ps_\nu\cmp\Ph_\nu$ we have
\[
  \Ps_\nup^*(N+P_\nup-Y_\nup) = N + \Qn_\nu + \Ps_\nup^*\Dl P_\nup.
\]
For $\Qn_\nu$ we have
\[
  \nn*{\Qn_\nu}_{s_\nu/4}
  \le 4 \nn*{(I-\Pi)Q_\nu}_{s_\nu/2}
  \le 4^n \e^{-r/2} \nn{Q_\nu}_\sn.
\]
For the other term we have
\[
  \nn*{\Ps_\nup^*\Dl P_\nup}_{s_\nu/4}
  \le 2 \nn*{\Dl P_\nup}_{s_\nu/2}
  \le \frac{2\e^r}{m_{\nup}} \nn{\Dl P_\nup}_m.
\]
Putting both estimates together and taking into account the definitins of $\th$ and $B$ as well as the fact that $s_{\nu+1} \le s_\nu/4$ we arrive at
\begin{align*}
  \nn{Q_\nup}_{s_\nup}
  &\le \nn*{\Qn_\nu}_{s_\nu/4} + \nn*{\Ps_\nup^*\Dl P_\nup}_{s_\nu/4} \vphantom{\frac{.}{m_\nu}}\\
  &\le \th \nn{Q_\nu}_\sn + \frac{B}{m_{\nup}} \nn{\Dl P_\nup}_m\\
  &\le B \sum_{\mu\le \nu} \frac{\th^{\nup-\mu}}{m_\mu} \rh_\mu
        + \frac{B}{m_\nup} \rh_\nup 
  = B\! \sum_{\mu\le \nup} \frac{\th^{\nup-\mu}}{m_\mu} \rh_\mu 
\end{align*}
as required.
---
Finally, consider $\Ps_\nup = I + \Psh_\nup$. We have
\[
  D\Psh_\nup = D\Psh_\nu\cmp\Ph_\nu + D\Ps_\nu\cmp\Ph_\nu\cd D\Phh_\nu.
\]
So in view of $\na{\Phh_\nu} \le \al s_\nu$ by estimate \eqref{edph} we have
\[
  \n*{D\Psh_\nup}_{s_\nup}
  \le \n*{D\Psh_\nu}_\sn + \n*{D\Ps_\nu}_\sn \n*{D\Phh_\nu}_{s_\nup}. 
\]
Setting $\et_\nu = 4\Dl_\nu \ep_\nu$ we thus get
\begin{align*}
  \n*{D\Psh_\nup}_{s_\nup}
  &\le \dl_\nu + (1+\dl_\nu)\et_\nu 
  \\
  &= (1+\dl_\nu)(1+\et_\nu) - 1 
   = \prod_{0\le\mu<\nup} (1+\et_\mu) - 1
   = \dl_\nup
\end{align*}
as required in the Iterative Lemma.
Moreover,
\[
  D\Ps_\nup-D\Ps_\nu
  = D\Ps_\nu\cmp\Ph_\nu-D\Ps_\nu + D\Ps_\nu\cmp\Ph_\nu \cd D\Phh_\nu.
\]
It is not difficult to see that
\[
  \n*{D^2\Ps_\nu}_0 
  \le \frac{4}{s_\nu} \n{D\Ps_\nu}_\sn
  \le \frac{8}{s_\nu}
\]
for $\ep$ sufficiently small. As $s_\nu K_\nu=r$ is sufficiently large, we conclude that
\mmxx
\begin{align*}
  \n{D\Ps_\nu\cmp\Ph_\nu-D\Ps_\nu}_0
  &\le \n*{D^2\Ps_\nu}_0 \n*{\Phh_\nu}_0 \\
  &\le \frac{8}{s_\nu} \cd \frac{4\Dl_\nu\ep_\nu}{K_\nu} 
  \le 4\Dl_\nu\ep_\nu.
\end{align*}
\[
  \n{D\Ps_\nu\cmp\Ph_\nu-D\Ps_\nu}_0
  \le \n*{D^2\Ps_\nu}_0 \n*{\Phh_\nu}_0 
  \le \frac{8}{s_\nu} \, \frac{4\Dl_\nu\ep_\nu}{K_\nu} 
  \le 4\Dl_\nu\ep_\nu.
\]
Similary, $\n*{D\Ps_\nu\cmp\Ph_\nu \cd D\Phh_\nu}_0 \le 4\Dl_\nu\ep_\nu$.
This completes the proof of the Iterative Lemma.
\end{proof}

\section{Convergence}

The convergence of the scheme described in the Iterative Lemma is obvious.
The modyfing terms $Y_\nu$ have a limit $Y$, the transformations $\Ps_\nu$ have a limit $\Ps$ in the \m{C^1}-norm, and $Q_\nu$ vanishes with respect to $\nn\cd_0$.
Hence, passing to the limit in equation
\[
  \Ps_\nu^*(N+P_\nu-Y_\nu) = N + Q_\nu
\]
we obtain
\[
  \Ps^*(N+P-Y) = N
\]
as promised.

\Newpage

\end{document}